\documentclass[12pt]{amsart}
\usepackage{fullpage}
\usepackage{longtable}

\usepackage[upright]{fourier}\usepackage{baskervald} 
\usepackage{dsfont} 

\usepackage{amsmath,amsthm,amssymb,amsfonts,amscd}
\usepackage{amsxtra}
\usepackage{eucal}
\usepackage{mathrsfs}
\usepackage{color, graphicx}
\usepackage[all]{xy}
\usepackage{hyperref}
\usepackage{pstricks,pst-plot}
\usepackage{caption,subcaption}
\usepackage{todonotes}
\usepackage{xcolor}
\usepackage{fancyvrb}  

\definecolor{forestgreen(traditional)}{rgb}{0.0, 0.27, 0.13}
\definecolor{forestgreen(web)}{rgb}{0.13, 0.55, 0.13}
\definecolor{airforceblue}{rgb}{0.36, 0.54, 0.66}

\hypersetup{
  colorlinks,
  citecolor=blue,
  linkcolor=forestgreen(web),
  urlcolor=magenta}

\newtheorem{thm}{Theorem}[section]
\newtheorem*{thm*}{Theorem}
\newtheorem*{thmB_again}{Theorem 1.2}
\newtheorem{lemma}[thm]{Lemma}
\newtheorem{proposition}[thm]{Proposition}
\newtheorem{corollary}[thm]{Corollary}

\theoremstyle{definition}

\newtheorem{definition}{Definition}
\newtheorem{question}{Question}
\newtheorem{remark}[thm]{Remark}

\newtheorem{convention and reminder}{Convention and Reminder}
\newtheorem{convention and remark}{Convention and Remark}
\newtheorem{definition and remark}{Definition and Remark}

\newtheorem{reminders and definition}{Reminders and Definition}

\newtheorem{notation and remarks}{Notation and Remarks}
\newtheorem{notation and remark}{Notation and Remark}
\newtheorem{example}[thm]{Example}

\newcommand \sat {\mathrm {sat}}

\newcommand \N {\ensuremath{\mathrm{\bf N}}}
\newcommand \ini {\ensuremath{\mathrm{in}}}

\newcommand \ch {\ensuremath{\mathrm{ch}}}
\newcommand \coker {\ensuremath{\mathrm{coker}}}
\newcommand\Proj{\operatorname{Proj}}

\newcommand\length{\operatorname{length}}
\newcommand\mult{\operatorname{mult}}
\newcommand\Tor{\operatorname{Tor}}

\newcommand\depth{\operatorname{depth}}

\newcommand\reg{\operatorname{reg}}

\newcommand\Supp{\operatorname{Supp}}

\newcommand\Sec{\operatorname{Sec}}
\newcommand\Tan{\operatorname{Tan}}
\newcommand\gin{\operatorname{gin}}
\newcommand\Gin{\operatorname{Gin}}

\newcommand\ND{\mathrm{ND}}
\newcommand\Sing{\operatorname{Sing}}

\newcommand\Ker{\operatorname{Ker}}

\newcommand \pd {\operatorname{proj.dim}}

\newcommand \st[1] {\stackrel{#1}{\rightarrow}}

\def\ds{\displaystyle}
\def\P{{\mathbb P}}
\def\Z{{\mathbb Z}}
\def\kk{{\Bbbk}}

\newcommand \tb[1]{\textcolor{blue}{#1}}

\newcommand \ti[1]{\textit{#1}}

\begin{document}

\title[Condition ${\mathrm ND}(\ell)$ and the first non-trivial strand of syzygies]{On the first non-trivial strand of syzygies of projective schemes and  Condition ${\mathrm ND}(\ell)$}
\author[J.\ Ahn, K.\ Han and S.\ Kwak] {Jeaman Ahn, Kangjin Han${}^{*}$, Sijong Kwak}
\address{Jeaman Ahn\\
Department of Mathematics Education, Kongju National University, 182, Shinkwan-dong, Kongju, Chungnam 314-701, Republic of Korea}
\email{jeamanahn@kongju.ac.kr}
\address{Kangjin Han\\
School of Undergraduate Studies,
Daegu-Gyeongbuk Institute of Science \& Technology (DGIST),
333 Techno jungang-daero, Hyeonpung-myeon, Dalseong-gun
Daegu 42988,
Republic of Korea}
\email{kjhan@dgist.ac.kr}
\address{Sijong Kwak\\
Department of Mathematics, Korea Advanced Institute of Science and Technology (KAIST), 373-1 Gusung-dong, Yusung-Gu, Daejeon, Republic of Korea}
\email{skwak@kaist.ac.kr}
\thanks{${}^{*}$ Corresponding author.}
\thanks{The first author was supported by Basic Science Research Program through the National Research Foundation of Korea (NRF) funded by the Ministry of Education, Science, and Technology (No.2019R1F1A1058684). The second author was supported by the POSCO Science Fellowship of POSCO TJ Park Foundation
and the DGIST Start-up Fund of the Ministry of Science, ICT and Future Planning (No.2016010066). The third author was supported by Basic Science Research Program through NRF funded by the Ministry of Science and ICT(No.2015R1A2A2A01004545).}

\date{\today}


\begin{abstract}
Let $X\subset\P^{n+e}$ be any $n$-dimensional closed subscheme. We are mainly interested in two notions related to syzygies: one is the property $\N_{d,p}~(d\ge 2, ~p\geq 1)$, which means that $X$ is $d$-regular up to $p$-th step in the minimal free resolution and the other is a new notion $\mathrm{ND}(\ell)$ which generalizes the classical ``being nondegenerate'' to the condition that requires a general finite linear section  not to be contained in any hypersurface of degree $\ell$. 

First, we introduce condition $\mathrm{ND}(\ell)$ and consider examples and basic properties deduced from the notion. Next we prove sharp upper bounds on the graded Betti numbers of the first non-trivial strand of syzygies, which generalize results in the quadratic case to higher degree case, and provide characterizations for the extremal cases. Further, after regarding some consequences of property $\N_{d,p}$, we characterize the resolution of $X$ to be $d$-linear arithmetically Cohen-Macaulay as having property $\N_{d,e}$ and condition $\mathrm{ND}(d-1)$ at the same time. From this result, we obtain a syzygetic rigidity theorem which suggests a natural generalization of syzygetic rigidity on $2$-regularity due to Eisenbud-Green-Hulek-Popescu to a general $d$-regularity.
\end{abstract}
\maketitle
\tableofcontents \setcounter{page}{1}

\section{Introduction}\label{sec_into} Since the foundational paper on syzygy computation by Green (\cite{G2}), there has been a great deal of interest and progress
in understanding the structure of the Betti tables of algebraic varieties during the past decades. In particular,
the first non-trivial linear strand starting from quadratic equations has been intensively studied by several authors
(\cite{C}, \cite{G2}, \cite{GL}, \cite{EGHP1,EGHP2}, \cite{EL}, \cite{HK2} etc.).

Let $X$ be any nondegenerate $n$-dimensional closed subscheme $X$ in a projective space $\P^{n+e}$ defined over an algebraically closed field $\kk$ of any characteristic and $R=\kk[x_0,\ldots,x_{n+e}]$. In this article, we are mainly interested in two notions related to syzygies of $X$. One notion is the property $\N_{d,p}~(d\ge 2, ~p\geq 1)$, which was first introduced in \cite{EGHP1} and means that $X$ is $d$-regular up to $p$-th step in the minimal free resolution. To be precise, $X$ is said to satisfy property $\N_{d,p}$ if the following condition holds:
$$\beta_{i,j}(X):=\dim_{\kk}\Tor^R_{i}(R/I_X, \kk)_{i+j}=0 \quad \text{ for }  i\leq p \text{ and } j\geq d.$$

The other one is a new notion \textit{condition $\mathrm{ND}(\ell)$}, which generalizes the classical ``being nondegenerate'' in degree one to cases of higher degrees. More precisely, it means that a general linear section $X\cap \Lambda$ is not contained in any hypersurface of degree $\ell$ of $\Lambda$, where $\Lambda$ is a general linear subspace of each dimension $\geq e$. So, for irreducible varieties the classical nondegenerate condition is equivalent to condition $\mathrm{ND}(1)$ by Bertini-type theorem. We give many examples and basic properties on condition $\ND(\ell)$.

\begin{figure}[!htb]
\centering
\begin{subfigure}{.44\textwidth}
\captionsetup{width=0.99\textwidth}
\texttt{
\begin{tabular}{l|cccccccccccccccccccccccc}
          & 0 & 1   &  $\cdots$    & p    &  p+1 & $\cdots$   &  e   & e+1   &  $\cdots$ & n+e     \\ \hline
    0    & 1 & -    & -      & -      &  -   & -     &- &   -  &   -&   -     \\
    1   & - & $\ast$  & $\ast$ & $\ast$  & $\ast$ &$\ast$ & $\ast$ &$\ast$ &$\ast$ &$\ast$ \\
\vdots    & \vdots  & \vdots & \vdots & \vdots & \vdots & \vdots &\vdots & \vdots & \vdots & \vdots \\
d-2 & -  & $\ast$ & $\ast$ & $\ast$ & $\ast$ & $\ast$ &$\ast$ &$\ast$ &$\ast$ &$\ast$ \\
    d-1    & -  & $\ast$ & $\ast$ & $\ast$ & $\ast$ & $\ast$ &$\ast$ &$\ast$ &$\ast$ &$\ast$ \\
    d    & -  & - & - & - & $\ast$ &$\ast$ &$\ast$ & $\ast$ &$\ast$ &$\ast$ \\
d+1  & - & -    & -      & -  & $\ast$ &$\ast$ & $\ast$ &$\ast$ &$\ast$ &$\ast$ \\
\vdots    & \vdots  & \vdots & \vdots & \vdots & \vdots & \vdots &\vdots & \vdots & \vdots & \vdots \\
    \end{tabular}}
\label{fig_Ndp}
\end{subfigure}
\hspace*{\fill}
\begin{subfigure}{.44\textwidth}
\captionsetup{width=0.99\textwidth}
\texttt{
\begin{tabular}{l|cccccccccccccccccccccccc}
          & 0 & 1   &  $\cdots$    &  $\cdots$ & e-1  &  e   & e+1   &  $\cdots$ & n+e     \\ \hline
   0    & 1 & -    & -      & -        & -     &- &   -  &   -&   -     \\
   1    & - & -    & -      & -        & -     &- &   -  &   -&   -     \\
\vdots    & \vdots  & \vdots  & \vdots & \vdots & \vdots & \vdots &\vdots & \vdots & \vdots \\
$\ell$-1    & - & -    & -      & -         & -     &- &   -  &   -&   -     \\
    $\ell$    & -  & $\ast$ & $\ast$  & $\ast$ & $\ast$ &$\ast$ & $\ast$ & $\ast$ & $\ast$ \\
    $\ell$+1    & - & $\ast$    & $\ast$       & $\ast$          & $\ast$      &$\ast$  &   $\ast$   & $\ast$ & $\ast$      \\
    $\ell$+2    & - & $\ast$     & $\ast$    &$\ast$        & $\ast$    &$\ast$ &   $\ast$   & $\ast$ &  $\ast$     \\
\vdots    & \vdots  & \vdots & \vdots & \vdots & \vdots & \vdots &\vdots & \vdots & \vdots \\
    \end{tabular}}
    \label{fig_ND(d-1)}
\end{subfigure}
\caption{Two typical Betti tables $\mathbb{B}(X)$ of $X\subset\P^{n+e}$ with property $\mathbf{N}_{d,p}$ and with condition $\mathrm{ND}(\ell)$. Note that the shape of $\mathbb{B}(X)$ with $\mathrm{ND}(\ell)$ is preserved under taking general hyperplane sections and general linear projections.}\label{fig_ND(d-1)}
\end{figure}

With this notion, we obtain a new angle to study syzygies of high degrees in the Betti table $\mathbb{B}(X)$. Especially, it turns out to be very effective to understand the first non-trivial $\ell$-th linear strand arising from equations of degree $\ell+1$ and also to answer many interesting questions which can be raised as compared to the classical quadratic case.

To review previous results for the quadratic case, let us begin by recalling the well known theorems due to Castelnuovo and Fano:
Let $X\subset \P^{n+e}$ be any ``nondegenerate'' irreducible variety.
\begin{itemize}
\item (Castelnuovo, 1889) $h^0(\mathcal I_{X}(2)) \le {e+1\choose 2}$ and $``="$ holds iff $X$ is a variety of minimal degree.
\item (Fano, 1894) Unless $X$ is a variety of minimal degree, $h^0(\mathcal I_{X}(2)) \le {e+1\choose 2}-1$ and $``="$ holds iff $X$ is a del Pezzo variety
(i.e. arithmetically Cohen-Macaulay and $\deg(X)=e+2$).
\end{itemize}

A few years ago, Han and Kwak developed an inner projection method to compare syzygies of $X$ with those of its projections by using the theory of mapping cone and partial elimination ideals. As applications, over any algebraically closed field $\kk$ of arbitrary characteristic, they proved the sharp upper bounds on the ranks of higher linear syzygies by quadratic equations, and characterized the extremal and next-to-extremal cases, which generalized the results of Castelnuovo and Fano \cite{HK2}:

\begin{itemize}
\item (Han-Kwak, 2015) $\beta_{i,1}(X)\leq i{e+1 \choose i+1}, ~~i\ge 1$
and the equality holds for some $1\le i \le e$ iff $X$ is a variety of minimal degree (abbr. VMD);
\item Unless $X$ is a variety of minimal degree, then $\beta_{i,1}(X)\leq i{e+1 \choose i+1}-{e\choose i-1} ~~\forall i\le e$
and the equality holds for some $1\le i\le e-1$ iff $X$ is a del Pezzo variety.
\end{itemize}
Thus, the theorem above by Han-Kwak can be thought of as a syzygetic characterization of varieties of minimal degree and del Pezzo varieties.

It is worth to note here that the condition $(I_X)_1=0$ (i.e. to be ``nondegenerate'') implies not only an upper bound for the number of quadratic equations $h^0(\mathcal I_{X}(2)) \le {e+1\choose 2}$ as we reviewed, but also on the degree of $X$ via the so-called `basic inequality' $\deg(X)\ge{e+1\choose 1}$. Thus, for `more' nondegenerate varieties, it seems natural to raise a question as follows: For any irreducible variety $X$ with $(I_X)_2=0$ (i.e. having no linear and quadratic forms vanishing on $X$),
\begin{center}
`` Does it hold that $h^0(\mathcal I_{X}(3)) \le {e+2\choose 3}$ and $\deg(X)\ge{e+2\choose 2}$? ''
\end{center}

But, there is a counterexample for this question : the Veronese surface $S\subset \P^4$ ($e=2$) i.e. an isomorphic projection of $\nu_2(\P^2)$, one of the Severi varieties classified by Zak, where $S$ has no quadratic equations on it, but $h^0(\mathcal I_{S}(3))=7 \nleq {2+2\choose 3}$ and $\deg(X)=4\ngeq{2+2\choose 2}$. One reason for the failure is that a general hyperplane section of $S$ sits on a quadric hypersurface while $S$ itself does not. It leads us to consider the notion of condition $\mathrm{ND}(\ell)$.

Under condition $\mathrm{ND}(\ell)$ it can be easily checked that the degree of $X$ satisfies the expected bound $\deg(X)\ge\binom{e+\ell}{\ell}$ (see Remark \ref{basic_properties}). Further, one can see that condition $\mathrm{ND}(\ell)$ is determined by the \textit{injectivity} of the restriction map $H^{0}(\mathcal{O}_{\Lambda}(\ell))\to H^{0}(\mathcal{O}_{X\cap\Lambda}(\ell))$ for a general point section $X\cap\Lambda$ which can happen in larger degree for a given $\ell$, while the problem on `imposing independent conditions on $\ell$-forms (or $\ell$-normality)' concerns surjectivity of the above map in degree at most $\binom{e+\ell}{\ell}$. The latter has been intensively studied in many works in the literature (see e.g. \cite{CHMN} and references therein), but the former has not been considered well.  

With this notion, we can also obtain sharp upper bounds on the numbers of defining equations of degree $\ell+1$ and the graded Betti numbers for their higher linear syzygies. As in the quadratic case, we prove that the extremal cases for these Betti numbers are only arithmetically Cohen-Macaulay (abbr. ACM) varieties with $(\ell+1)$-linear resolution (we call a variety $X\subset \P^N$ ACM if its homogeneous coordinate ring $R_X$ is arithmetically Cohen-Macaulay i.e. $\depth(R_X)=\dim(X)+1$).

Now, we present our first main result.
\begin{thm}\label{thm A}
Let $X$ be any closed subscheme of codimension $e$ satisfying condition $\mathrm{ND}(\ell)$ for some $\ell\ge 1$ in $\P^{n+e}$ over an algebraically closed field $\kk$ with $\ch(\kk)=0$.
Then, we have
\begin{itemize}
\item[(a)] $\beta_{i,\ell}(X)\le{i+\ell-1\choose \ell}{e+\ell\choose i+\ell}$ for all $i \ge 1$.
\item[(b)] The following are equivalent:
\begin{itemize}
\item[(i)] $\beta_{i,\ell}(X)={i+\ell-1\choose \ell}{e+\ell \choose i+\ell}$ for all $i \ge 1$;
\item[(ii)] $\beta_{i,\ell}(X)={i+\ell-1\choose \ell}{e+\ell\choose i+\ell}$ for some $i$ among $1\le i\le e$;
\item[(iii)] $X$ is arithmetically Cohen-Macaulay with $(\ell+1)$-linear resolution.
\end{itemize}
In particular, if $X$ satisfies one of equivalent conditions then $X$ has a minimal degree ${e+\ell\choose \ell}$.
\end{itemize}
\end{thm}

We would like to note that if $\ell=1$, then this theorem recovers the previous results on the linear syzygies by quadrics for the case of integral varieties (see also Remark \ref{thm A_recover_HK}). In general, the set of closed subschemes satisfying $\mathrm{ND}(1)$ is much larger than that of \textit{nondegenerate} irreducible varieties (see \cite[section 1]{AH} for details). Furthermore, a closed subscheme $X$ (with possibly many components) has condition $\mathrm{ND}(\ell)$ if so does the top-dimensional part of $X$. Note that the Betti table $\mathbb{B}(X)$ is usually very sensitive for addition some components to $X$ (e.g. when we add points to a rational normal curve, Betti table can be totally changed, see e.g. \cite[example 3.10]{AK2}). But condition $\mathrm{ND}(\ell)$ has been still preserved under such addition of low dimensional components (thus, we could make many examples with condition $\mathrm{ND}(\ell)$ in this way).

On the other hand, if $X$ satisfies property $\mathbf{N}_{d,e}$, then the degree of $X$ is at most $\binom{e+d-1}{d-1}$ and the equality happens only when $X$ has ACM $d$-linear resolution. We prove this by establishing a syzygetic B\'ezout Theorem (Theorem~\ref{thm:20160115-01}), a geometric implication of property $\mathbf{N}_{d,p}$ using projection method. We also investigate an effect of $\mathbf{N}_{d,p}$ on loci of $d$-secant lines (Theorem~\ref{Loci_d_secant_lines}).

Furthermore, if two notions - condition $\mathrm{ND}(d-1)$ and property $\mathbf{N}_{d,e}$ on $X$ - meet together, then the degree of $X$ should be equal to $\binom{e+d-1}{d-1}$ and $X$ has ACM $d$-linear resolution (in particular, $X$ is $d$-regular). From this point of view, we can obtain another main result, a syzygetic rigidity for $d$-regularity as follows:

\begin{thm}[Syzygetic rigidity for $d$-regularity]\label{syz_rigid} Let $X$ be any algebraic set of codimension $e$ in $\P^{n+e}$ satisfying condition $\mathrm{ND}(d-1)$ for $d\ge2$. If $X$ has property $\N_{d,e}$, then $X$ is $d$-regular (more precisely, $X$ has ACM $d$-linear resolution).
\end{thm}

Note that if $d=2$, for nondegenerate algebraic sets this theorem recovers the syzygetic rigidity for $2$-regularity due to Eisenbud-Green-Hulek-Popescu (\cite[corollary 1.8]{EGHP1}) where the condition $\mathrm{ND}(1)$ was implicitly used. In \cite{EGHP1}, the rigidity for $2$-regularity was obtained using the classification of so-called `small' schemes in the category of algebraic sets in \cite{EGHP2}. But, for next $3$ and higher $d$-regular algebraic sets, it seems out of reach to get such classifications at this moment. From this point of view, Theorem \ref{syz_rigid} is a natural generalization and gives a more direct proof for the rigidity.  

We would like to also remark that for a generalization of this syzygetic rigidity into higher $d$, one needs somewhat a sort of `higher nondegeneracy condition' such as the condition $\mathrm{ND}(\ell)$, because there exist some examples where Theorem \ref{syz_rigid} does not hold without condition $\mathrm{ND}(\ell)$ even though the given $X$ is an irreducible variety and there is no forms of degree $\ell$ vanishing on $X$ (see Figure \ref{ND(l)meetsNdp} and Example \ref{non_ex_syz_rigidity}).


\begin{figure}[!htb]
\centering
\begin{subfigure}{.44\textwidth}
\captionsetup{width=0.99\textwidth}

\texttt{
\begin{tabular}{l|cccccccccccccccccccccccc}
          & 0 & 1   &  $\cdots$  &  $\cdots$ & e-1   &  e   & e+1   &  $\cdots$ & N     \\ \hline
 0    & 1 & -    & -      & -       & -     &- &   -  &   -&   -     \\
 1    & - & -    & -      & -        & -     &- &   -  &   -&   -     \\
\vdots    & \vdots  & \vdots  & \vdots & \vdots & \vdots & \vdots &\vdots & \vdots & \vdots \\
d-2    & - & -    & -      & -         & -     &- &   -  &   -&   -     \\
    d-1    & -  & $\ast$ & $\ast$ & $\ast$ & $\ast$ &$\ast$ &$\ast$ &$\ast$ &$\ast$ \\
    d    & -  & - & - & - & - & - & $\ast$ &$\ast$ &$\ast$ \\
  d+1    & -  & - & - & - & - & - & $\ast$ &$\ast$ &$\ast$ \\
 \vdots    & \vdots  & \vdots & \vdots & \vdots & \vdots & \vdots &\vdots & \vdots & \vdots \\
    \end{tabular}}
\caption{$\mathbf{N}_{d,e}$ meets only $(I_X)_{\le d-1}=0$ }
\label{fig_No_FormAndNde}
\end{subfigure}
\hspace*{\fill}
\begin{subfigure}{.44\textwidth}
\captionsetup{width=0.99\textwidth}

\texttt{
\begin{tabular}{l|cccccccccccccccccccccccc}
          & 0 & 1   &  $\cdots$    &  $\cdots$ & e-1  &  e   & e+1   &  $\cdots$ & N     \\ \hline
   0    & 1 & -    & -      & -        & -     &- &   -  &   -&   -     \\
   1    & - & -    & -      & -        & -     &- &   -  &   -&   -     \\
\vdots    & \vdots  & \vdots  & \vdots & \vdots & \vdots & \vdots &\vdots & \vdots & \vdots \\
d-2    & - & -    & -      & -         & -     &- &   -  &   -&   -     \\
    d-1    & -  & $\ast$ & $\ast$  & $\ast$ & $\ast$ &$\ast$ & - & - & - \\
    d    & - & -    & -      & -         & -     &- &   -  &   -&   -     \\
    d+1    & - & -    & -      & -         & -     &- &   -  &   -&   -     \\
\vdots    & \vdots  & \vdots & \vdots & \vdots & \vdots & \vdots &\vdots & \vdots & \vdots \\
    \end{tabular}}

\caption{$\mathbf{N}_{d,e}$ meets condition $\mathrm{ND}(d-1)$ }
\end{subfigure}
\caption{condition $\mathrm{ND}(d-1)$ and property $\mathbf{N}_{d,e}$ implies ACM $d$-linear resolution in the category of algebraic sets.}
\label{ND(l)meetsNdp}
\end{figure}

In the final section \ref{sect_question}, we present relevant examples and more consequences of our theory (see e.g. Corollary \ref{ND_hvec}) and raise some questions for further development.

\textbf{Acknowledgement} We are grateful to Frank-Olaf Schreyer for suggesting Example \ref{non_ex_syz_rigidity} and Ciro Ciliberto for reminding us of using `lifting theorems'. The second author also wishes to thank Aldo Conca, David Eisenbud for their questions and comments on the subject and Hailong Dao, Matteo Varbaro for useful discussions on $h$-vectors and condition $\mathrm{ND}(\ell)$.
\bigskip

\section{Condition $\mathrm{ND}(\ell)$ and Syzygies}

\subsection{Condition $\mathrm{ND}(\ell)$ : basic properties and examples}

Throughout this section, we assume that the base field is algebraically closed and $\ch(\kk)=0$ (see Remark \ref{remk_ch_p} for finite characteristics). 

As before, let $X$ be a $n$-dimensional closed subscheme of codimension $e$ in $\mathbb P^{N}$ over $\kk$. Let $I_X$ be $\bigoplus_{m=0}^{\infty} H^0(\mathcal I_{X/\P^N}(m))$, the defining ideal of $X$ in the polynomial ring $R=\kk[x_0, x_1 ,\ldots, x_{N}]$. We mean (co)dimension and degree of $X\subset\P^N$ by the definition deduced from the Hilbert polynomial of $R/I_X$.

Let us begin this study by introducing the definition of condition $\mathrm{ND}(\ell)$ as follows:
\begin{definition}[Condition $\mathrm{ND}(\ell)$]\label{def_NDk}
Let $\kk$ be any algebraically
closed field. We say that a closed subscheme $X\subset\P_\kk^N$ satisfies \textit{condition} $\mathrm{ND}(\ell)$ if
\[H^0(\mathcal{I}_{X\cap \Lambda/\Lambda}(\ell))=0\quad\textrm{for a \textit{general} linear section $\Lambda$ of each dimension $\geq e$}.\]
We sometimes call a subscheme with condition $\mathrm{ND}(\ell)$ a \textit{$\mathrm{ND}(\ell)$-subscheme} as well.
\end{definition}

\begin{remark}\label{basic_properties} We would like to make some remarks on this notion as follows: 
\begin{itemize}
\item [(a)] First of all, if $X\subset\P^N$ satisfies condition $\mathrm{ND}(\ell)$, then every general linear section of $X\cap \Lambda$ also has the condition (i.e. condition $\mathrm{ND}(\ell)$ is preserved under taking general hyperplane sections). Further, from the definition, condition $\mathrm{ND}(\ell)$ on $X$ is completely determined by a general point section of $X$.
\item[(b)] (Basic degree bound) If $X$ is a closed subscheme of codimension $e$ in $\P^{n+e}$ satisfying condition $\mathrm{ND}(\ell)$, then from the sequence $0\to H^{0}(\mathcal{I}_{X\cap\Lambda/\Lambda}(\ell))\to H^{0}(\mathcal{O}_{\Lambda}(\ell))\to H^{0}(\mathcal{O}_{X\cap\Lambda}(\ell))$ it can be easily proved that
$\deg(X)\ge {e+\ell\choose \ell}$.
\item[(c)] A general linear projection of $\mathrm{ND}(\ell)$-subscheme is also an $\mathrm{ND}(\ell)$-subscheme.
\item [(d)] Any nondegenerate variety (i.e. irreducible and reduced) satisfies condition $\mathrm{ND}(1)$ due to Bertini-type theorem (see. e.g. \cite[lemma 5.4]{E2}). 
\item [(e)] If a closed subscheme $X\subset \mathbb P^N$ has top dimensional components satisfying $\ND{(\ell)}$, then $X$ also satisfies condition $\ND(\ell)$ whatever $X$ takes as a lower-dimensional component.
\item [(f)] (Maximal $\mathrm{ND}$-index) From the definition, it is easy to see that 
\begin{center}
`X: not satisfying condition $\mathrm{ND}(\ell)$ $\Rightarrow$ X: neither having $\mathrm{ND}(\ell+1)$.'
\end{center} 
Thus, it is natural to regard a notion like
\begin{equation}\label{maxND}
\mathrm{index}_{ND}(X):=\max\{ \ell\in\Z_{\ge0} : \textrm{$X$ satisfies condition $\mathrm{ND}(\ell)$} \}
\end{equation}
which is a new projective invariant of a given subscheme $X\subset\P^N$.
\item[(g)] From the viewpoint (a), one can re-state the definition of condition $\mathrm{ND}(\ell)$ as the \textit{injectivity} of the restriction map $H^{0}(\mathcal{O}_{\Lambda}(\ell))\to H^{0}(\mathcal{O}_{X\cap\Lambda}(\ell))$ for a general point section $X\cap\Lambda$ , while many works in the literature have focused on \textit{surjectivity} (or imposing independent conditions) to study dimensions of linear systems in relatively small degree.
\end{itemize}
\end{remark}

\begin{example}\label{basic_examples} We list some first examples achieving condition $\mathrm{ND}(\ell)$.
\begin{itemize}
\item [(a)] If $X\subset \P^{n+e}$ is an ACM subscheme with $ H^0(\mathcal{I}_{X}(\ell))=0$, then $X$ is an $\mathrm{ND}(\ell)$-subscheme.
\item[(b)] Every linearly normal curve with no quadratic equation is a $\mathrm{ND}(2)$-curve. Further, a variety $X$ is $\ND(2)$ if a general curve section $X\cap\Lambda$ is linearly normal. 
\item[(c)] (From a projection of Veronese embedding) We can also find examples of non-ACM $\ND(\ell)$-variety using projections. For instance, if we consider the case of $v_3(\P^{2})\subset\P^{9}$ and its general projection into $\P^4$ (say $\pi(v_3(\P^{2}))$), then $\deg \pi(v_3(\P^{2}))=9\ge {2+2\choose 2}$ and all the quadrics disappear after this projection. This is a $\ND(2)$-variety by Proposition \ref{codim2_ND_k} (see also Remark \ref{NDk_in_large}).
\end{itemize}
\end{example}

In general, it is not easy to determine whether a given closed subscheme $X$ satisfies condition $\mathrm{ND}(\ell)$ or not. The following proposition tells us a way to verify condition $\mathrm{ND}(\ell)$ by aid of computation the generic initial ideal of $X$ (see e.g. \cite[section 1]{BCR} and references therein for the theory of generic initial ideal and Borel fixed property). 

In what follows, for a homogeneous ideal $I$ in $R$, we denote by $\Gin(I)$ the generic initial ideal of $I$ with respect to the \textit{degree reverse lexicographic order}.

\begin{proposition}[A characterization of condition $\mathrm{ND}(\ell)$]\label{lem:2017-06-15}
 Let $X$ be a closed subscheme of codimension $e$ in $\mathbb P^{n+e}$. Then the followings are equivalent.
 \begin{itemize}
  \item[(a)] $X$ satisfies condition $\mathrm{ND}(\ell)$.
  \item[(b)] $\Gin(I_X) \subset (x_0,\ldots, x_{e-1})^{\ell+1}$.
 \end{itemize} 
\end{proposition}
\begin{proof}
Let $\Lambda$ be a general linear space of dimension $e$ and let $\Gamma$ be the zero-dimensional intersection of $X$ with $\Lambda$.\\
$(a)\Rightarrow(b)$ : For a monomial $T \in \Gin(I_X)$, decompose $T$ as a product of two monomials $N$ and $M$ such that
$$\Supp(N)\subset \{x_0,\cdots, x_{e-1}\} \quad \text{ and}\quad  \Supp(M)\subset \{x_e,\cdots, x_{n+e}\}.$$
By the Borel fixed property, we see that $Nx_e^{\deg(M)} \in \Gin(I_X)$.  Then, it follows from \cite[theorem~2.1]{AH} that
\[\Gin(I_{\Gamma/\Lambda})=\left[\frac{(\Gin(I_X), x_{e+1},\ldots, x_{n+e})}{(x_{e+1},\ldots, x_{n+e})}\right]^{\sat}=\left[\frac{(\Gin(I_X), x_{e+1},\ldots, x_{n+e})}{(x_{e+1},\ldots, x_{n+e})}\right]_{x_{e}\to 1},\]
which implies $N \in \Gin(I_{\Gamma/\Lambda})$. By the assumption that $X$ satisfies $\mathrm{ND}(\ell)$, we see that $\deg(N)\geq \ell+1$, and thus $N\in (x_0,\ldots, x_{e-1})^{\ell+1}$. Therefore $T=NM \in  (x_0,\ldots, x_{e-1})^{\ell+1}$
as we wished.\\
$(a)\Leftarrow(b)$: Conversely, assume that $\Gin(I_X) \subset (x_0,\ldots, x_{e-1})^{\ell+1}$.  Then,
\[\Gin(I_{\Gamma/\Lambda})=\left[\frac{(\Gin(I_X), x_{e+1},\ldots, x_{n+e})}{(x_{e+1},\ldots, x_{n+e})}\right]^{\sat}\subset \left[\frac{((x_0,\ldots, x_{e-1})^{\ell+1}, x_{e+1},\ldots, x_{n+e})}{(x_{e+1},\ldots, x_{n+e})}\right]_{x_{e}\to 1}.\]
Note that the rightmost ideal is identified with the ideal $(x_0,\ldots, x_{e-1})^{\ell+1}$ in the polynomial ring $\kk[x_0,\ldots, x_{e}]$. Therefore $(I_{\Gamma/\Lambda})_{\ell}=0$ and thus $X$ satisfies condition $\mathrm{ND}(\ell)$.
\end{proof}

Beyond the first examples in Examples \ref{basic_examples}, one can raise a question as `Is there a higher-dimensional $\ND(\ell)$-variety $X$ which is linearly normal (i.e. not coming from isomorphic projections) but also non-ACM?'. We can construct such an example as a toric variety which is $3$-dimensional and has depth 3 as follows.

\begin{example}[A linearly normal and non-ACM $\ND(3)$-variety]\label{ex:20150626}
Consider a matrix
$$
A=\left[\begin{array}{rrrrrrrrrrr}
3&-5& 4 & 0 & 0 & 0\\
1& 0& 0 & 1 & 0 & 0\\
0& 1& 0 & 0 & 1 & 0\\
0& 0& 1 & 0 & 0 & 1\\
\end{array}
\right]
$$
and we consider the toric ideal induced by the matrix $A$. Using \texttt{Macaulay 2} \cite{M2}, we compute the defining ideal as
$$
\begin{array}{llllllllllll}
I_A&=& (x_1x_2^2x_3-x_0x_4x_5^2,\,\,  x_2x_3^3x_4-x_0^3x_1x_5,\,\,   x_0^2x_1^2x_2-x_3^2x_4^2x_5,\\
     & &    x_2^3x_3^4-x_0^4x_5^3,\,\,   x_0x_1^3x_2^3-x_3x_4^3x_5^3,\,\,   x_0^5x_1^3-x_3^5x_4^3,\,\,
      x_1^4x_2^5-x_4^4x_5^5).
\end{array}
$$
Then the generic initial ideal of $I_A$ with respect to degree reverse lexicographic order is
$$\mathrm{Gin}(I_A)=(x_0^4,\,\, x_0^3x_1^2,\,\, x_0^2x_1^3,\,\, x_0x_1^5,\,\, x_1^6,\,\, x_0x_1^4x_2^2,\,\, x_1^5x_2^2,\,\, x_0^3x_1x_2^4,\,\, x_0^2x_1^2x_2^5)~.$$
Hence, $I_A$ defines a $3$-dimensional toric variety $X\subset \mathbb P^5$ with $\mathrm{depth}(X)=3$, which satisfies condition $\mathrm{ND}(\ell)$ for $\ell\le3$ by Proposition \ref{lem:2017-06-15}. Note that $I_A$ is linearly normal but not ACM.
\end{example}

Finally, we would like to remark that condition $\mathrm{ND}(\ell)$ is expected to be generally satisfied in the following manner.

\begin{remark}[$\mathrm{ND}(\ell)$ in a relatively large degree]\label{NDk_in_large}
 For a given codimension $e$, fixed $\ell$, and any \ti{general} closed subscheme $X$ in $\P^{n+e}$, it is expected that 
\begin{align}\label{ND_k_expect}
&X\rightarrow \mathrm{ND}(\ell)\quad\textrm{as \quad $\deg(X)\rightarrow\infty$}
\end{align}
under the condition $H^0(\mathcal{I}_{X/\P^{n+e}}(\ell))=0$ and  exceptional cases do appear with some special geometric properties (e.g. such as projected Veronese surface), because the failure of $\mathrm{ND}(\ell)$ means that any general point section $X\cap\Lambda$ sits in a hypersurface of degree $\ell$, which is not likely to happen for a sufficiently large $\deg(X)$. For instance, the `expectation' (\ref{ND_k_expect}) can have an explicit form in case of codimension two in the following proposition (see Section \ref{sect_question} for further discussion).
\end{remark}

\begin{proposition}[$\mathrm{ND}(\ell)$ in codimension two]\label{codim2_ND_k}
Let  $X\subset\P^N$ be any nondegenerate integral variety of codimension two over an algebraically closed field $\kk$ with $\ch(\kk)=0$. Say $d=\deg(X)$. Suppose that $H^0(\mathcal{I}_{X/\P^N}(\ell))=0$ for some $\ell\ge2$. Then, any such $X$ satisfies condition $\mathrm{ND}(\ell)$ if $d>\ell^2+1$.
\end{proposition}

\begin{proof}[Proof of Proposition \ref{codim2_ND_k}]
For the proof, we would like to recall a result for the `lifting problem' (for the literature, see e.g. \cite{CC, Bo15} and references therein) as follows:\\

\textit{
``Let  $X\subset\P^N$ be any nondegenerate reduced irreducible scheme of codimension two over an algebraically closed field $\kk$ with $\ch(\kk)=0$ and let $X_H$ be the general hyperplane section of $X$. Suppose that $X_H$ is contained in a hypersurface of degree $\ell$ in $\P^{N-1}$ for some $\ell\ge2$. If  $d>\ell^2+1$, then $X$ is contained in a hypersurface of degree $\ell$ in $\P^{N}$.''
}\\

Say $n=\dim(X)$ and suppose that $X\subset\P^N$ does not satisfy $\mathrm{ND}(\ell)$. Then for some $r$ with $2\le r\le n+1$, the $(r-2)$-dimensional general linear section of $X$, $X\cap\Lambda^r$ lies on a hypersurface of degree $\ell$ in $\Lambda^{r}$ (i.e. $H^0(\mathcal{I}_{X\cap \Lambda^r/\Lambda^r}(\ell))\neq0$). By  above lifting theorem, this implies $H^0(\mathcal{I}_{X\cap \Lambda^{r+1}/\Lambda^{r+1}}(\ell))\neq0$ for the $(r-1)$-dimensional general linear section $X\cap \Lambda^{r+1}$. By repeating the argument, we obtain that $H^0(\mathcal{I}_{X/\P^N}(\ell))\neq0$, which is a contradiction.
\end{proof}

\begin{example}[General curves in $\P^3$]\label{sp_curve_ND_k} Suppose that $C\subset \mathbb P^3$ be a general curve of degree $d\geq g+3$ with non-special line bundle $\mathcal O_C(1)$, where $g$ is the genus of $C$. When $g\geq 3$, then by the maximal rank theorem due to Ballico-Ellia \cite{BE}, the natural restriction map
$$ H^0(\mathcal O_{\mathbb P^3}(2)) \to H^0(\mathcal O_C(2))$$
is injective. So there is no quadric containing $C$. Further, from Proposition~\ref{codim2_ND_k} we see that a general point section $C\cap H$ also has no quadric. Thus $C$ satisfies condition $\mathrm{ND}(2)$. In a similar manner, we can show that if $g\geq 8$ then such curve satisfies $\mathrm{ND}(3)$ and in general it has condition $\mathrm{ND}(\ell)$ in case of $d\ge\max\{g+3,\ell^2+2\}$. 
\end{example}

\subsection{Sharp upper bounds on Betti numbers of the first non-trivial strand}

From now on, we proceed to prove Theorem~\ref{thm A}, which is one of our main results.\\

\noindent {\bf Theorem~1.1 (a) }\label{pro1:2017-06-15}
Let $X$ be any closed subscheme of codimension $e$ in $\P^{n+e}$ satisfying condition $\mathrm{ND}(\ell)$ for some $\ell\ge1$ and let $I_{X}$ be the (saturated) defining ideal of $X$. Then we have
\begin{equation}\label{eq3:2017-06-15}
\beta_{i,\ell}(X)\leq {i+\ell-1\choose \ell}{e+\ell\choose i+\ell} \quad \text{ for all } i\ge 1.
\end{equation}

\begin{proof}[A proof of Theorem~1.1~(a) :]
First, recall that by \cite[corollary~1.21]{G} we have
\begin{equation}\label{EQ1: 2017-06-15}
\beta_{i,j}(X)\leq \beta_{i,j}(R/\Gin(I_X)) \quad \text{ for all }i, j\geq 0.
\end{equation}
By the assumption that $X$ satisfies condition $\mathrm{ND}(\ell)$ for a given $\ell>0$, we see that $\mathrm{Gin}(I_X)_d=0$ for $d\leq \ell$. Moreover, by Lemma~\ref{lem:2017-06-15}, we have
\begin{equation}\label{eq2:2017-06-15}
\Gin(I_X)\subset (x_0,\ldots, x_{e-1})^{\ell+1}.
\end{equation}

For a monomial ideal $I$, we write $\mathcal G(I)$ for the set of minimal monomial generators and $\mathcal G(I)_{j+1}$ for the subset of degree $j+1$ part. We denote $\max\{a: k_a>0\}$ for a given monomial $T=x_0^{k_0}\cdots x_n^{k_n}$ by $\max(T)$. Then, for any Borel fixed ideal $J\subset R$ we have a formula as
\begin{equation}\label{EK_betti_formula}
\beta_{i,j}(R/J)=\sum_{T\in{\mathcal G (J)}_{j+1}}{\max(T) \choose i-1} \quad \textrm{for every } i ,~j~
\end{equation}
from the result of Eliahou-Kervaire (see e.g. \cite[theorem~2.3]{AH}).

(i) Let $0\leq i\leq e$. Consider the ideal $J_0=(x_0,\cdots,x_{e-1})^{\ell+1}$ which is Borel-fixed. We see that $J_0$ is generated by the maximal minors of $(\ell+1)\times (\ell+e)$ matrix
$$
\left(
\begin{array}{cccccccccccccccc}
 x_0 & x_1 & \cdots & x_{e-1}    & 0     & \cdots & 0 & 0\\
 0    & x_0 &  x_1    & \cdots & x_{e-1} & 0        & \cdots & 0\\
 & &      & \cdots & &         &  & \\
 0    & \cdots & 0    & x_0 & x_1 & x_2        & \cdots & x_{e-1}\\
\end{array}
\right)\quad.
$$
So, the graded Betti numbers of $R/J_0$ are those given by the Eagon-Northcott resolution of the maximal minors of a generic matrix of size $(\ell+1)\times (\ell+e)$ (see \cite[remark~2.11]{GHM}). This implies that
\begin{equation}\label{eq4:2017-06-15}
\beta_{i, \ell}(R/J_0)= {i+\ell-1\choose \ell}{e+\ell\choose i+\ell}.
\end{equation}
By relation (\ref{eq2:2017-06-15}), we see $\mathcal{G}(\Gin(I_X))_{\ell+1}\subset \mathcal{G}(J_0)_{\ell+1}$. So,  above formula (\ref{EK_betti_formula}) implies $\beta_{i,\ell}(R/\Gin(I_X)) \leq \beta_{i, \ell}(R/J_0)$.
Consequently, for each $0\leq i\leq e$ we conclude that
$$\beta_{i,\ell}(X)\leq \beta_{i,\ell}(R/\Gin(I_X)) \leq \beta_{i, \ell}(R/J_0)= {i+\ell-1\choose \ell}{e+\ell\choose i+\ell}~, $$
as we wished.\\

(ii) Let $e< i$. By \eqref{eq2:2017-06-15}, we see that if $T\in \mathrm{Gin}(I_X)_{\ell+1}$ then $\max(T)\leq e-1$. Then, from (\ref{EK_betti_formula}) it follows
$$\beta_{i,\ell}(R/\Gin(I_X))=\sum_{T\in{\mathcal G (\Gin(I_X))}_{\ell+1}}{\max(T) \choose i-1}=0 \quad\quad \text{ for all } i>e.$$
Hence, we get $\beta_{i,\ell}(X)=0$ by \eqref{EQ1: 2017-06-15}.
\end{proof}

\noindent{\bf Theorem~1.1 (b) }\label{main_thm}
Let $X$ be any closed subscheme of codimension $e$ in $\P^{n+e}$ satisfying condition $\mathrm{ND}(\ell)$ for some $\ell\ge1$ and let $I_{X}$ be the (saturated) defining ideal of $X$. Then, the followings are all equivalent.
\begin{itemize}
\item[(i)] $\displaystyle\beta_{i,\ell}(X)={i+\ell-1\choose \ell}{e+\ell\choose i+\ell}$ for all $1\le i \le e$;\\[1ex]
\item[(ii)] $\displaystyle\beta_{i,\ell}(X)={i+\ell-1\choose \ell}{e+\ell\choose i+\ell}$ for some $1\le i\le e$;\\[1ex]
\item[(iii)] $\mathrm{Gin}(I_X)=(x_0,x_1,\cdots,x_{e-1})^{\ell+1}$.\\[1ex]
\item[(iv)] $X$ is an ACM variety with $(\ell+1)$-linear resolution;
\end{itemize}
In this case, $X$ has minimal degree, i.e. $\deg(X)={e+\ell\choose \ell}$.

\begin{proof}[A proof of Theorem~1.1~(b) :]
(i) $\Rightarrow$ (ii) : This is trivial.

(ii) $\Rightarrow$ (iii) : Suppose that there exists an index $i$ such that $1\le i\le e$ and $\beta_{i,\ell}(X)={i+\ell-1\choose \ell}{e+\ell\choose i+\ell}$. Recall that $J_0=(x_0,\cdots,x_{e-1})^{\ell+1}$ has the Borel fixed property.  By (\ref{EK_betti_formula}), we have
\begin{align*}
\beta_{i,\ell}(R/J_0)&=\sum_{T\in{\mathcal G(J_0)}_{\ell+1}}{\max(T) \choose i-1}=\sum_{j=i-1}^{e-1}\binom{j}{i-1}\bigg| \{ T\in \mathcal G(J_0)_{\ell+1}\mid \max(T)=j\}\bigg| \\
                        &=\sum_{j=i-1}^{e-1}\binom{j}{i-1}\dim_{\kk} x_j\cdot \kk[x_0, \cdots, x_j]_{\ell}=\sum_{j=i-1}^{e-1}\binom{j}{i-1}\binom{j+\ell}{\ell}.
\end{align*}
Hence we see from \eqref{eq4:2017-06-15} that the following binomial identity holds:
\begin{equation}\label{bino_id_from_betti}
{i+\ell-1\choose \ell}{e+\ell\choose i+\ell}=\sum_{j=i-1}^{e-1}\binom{j}{i-1}\binom{j+\ell}{\ell}~.
\end{equation}
By the assumption that $\beta_{i,\ell}(X)={i+\ell-1\choose \ell}{e+\ell\choose i+\ell}=\beta_{i,\ell}(R/J_0)$ and the binomial identity (\ref{bino_id_from_betti}), we have
\begin{align*}
\beta_{i,\ell}(R/\Gin(I_X))&=\sum_{T\in{\mathcal G(\Gin(I_X))}_{\ell+1}}{\max(T) \choose i-1}=\sum_{j=i-1}^{e-1}\binom{j}{i-1}\bigg| \{ T\in \mathcal G(\Gin(I_X))_{\ell+1}\mid \max(T)=j\}\bigg| \\
                                    &\leq\sum_{j=i-1}^{e-1}\binom{j}{i-1}\dim_{\kk} x_j\cdot \kk[x_0, \cdots, x_j]_{\ell}=\sum_{j=i-1}^{e-1}\binom{j}{i-1}\binom{j+\ell}{\ell}=\beta_{i,\ell}(R/I_X).
\end{align*}
Thus, by the cancellation principle \eqref{EQ1: 2017-06-15}, we conclude that $\beta_{i,\ell}(R/\Gin(I_X))=\beta_{i,\ell}(R/I_X)$. This implies that, for each $j$ with $i-1\le j\le e-1$,
$$ \{ T\in \mathcal G(\Gin(I_X))_{\ell+1}\mid \max(T)=j\}= x_j\cdot \kk[x_0, \cdots, x_j]_{\ell}~.$$

In particular, when $j=e-1$, we obtain that $x_{e-1}^{\ell+1}\in \Gin(I_X)$ and it follows from Borel fixed property that
$$\mathrm{Gin}(I_X)_{\ell+1}=(J_0)_{\ell+1}.$$
Now, since $X$ satisfies condition ${\rm ND}(\ell)$, by Lemma \ref{lem:2017-06-15} we have that $\mathrm{Gin}(I_X)\subset J_0$. Because $J_0$ is generated in degree $\ell+1$, this implies that $\mathrm{Gin}(I_X)= J_0$.

(iii) $\Rightarrow$ (iv) Note that if $\mathrm{Gin}(I_X)=(x_0,\cdots,x_{e-1})^{\ell+1}$, then $R/\Gin(I_X)$ has $\ell$-linear resolution. By Cancellation principle \cite[corollary~1.12]{G}, the minimal free resolution of $I_X$ is obtained from that of $\Gin(I_X)$ by canceling some adjacent terms of the same shift in the free resolution. This implies that the betti table of $R/I_X$ are the same as that of $R/\Gin(I_X)$, because $R/\Gin(I_X)$ has $\ell$-linear resolution. This means $R/I_X$ is arithmetically Cohen-Macaulay with $\ell$-linear resolution.

(iv) $\Rightarrow$ (i) This follows directly from \cite[proposition~1.7]{EG}.
\end{proof}

\begin{remark}\label{thm A_recover_HK}
For the case of $\ell=1$, Theorem~\ref{thm A} was proved in \cite{HK2} for any nondegenerate variety $X$ over any algebraically closed field (recall that every nondegenerate variety satisfies $\ND(1)$). Thus, this theorem is a generalization of the previous result to cases of $\ell\ge2$.

Further, we would also like to remark that for $\ell=1$ a given $X$ satisfies all the consequences of Theorem~\ref{thm A} (b) once the degree inequality $\deg(X)\ge {e+\ell\choose \ell}$ attains equality (i.e. the case of classical minimal degree), since they are all 2-regular and arithmetically Cohen-Macaulay. But, for higher $\ell\ge2$, this is no more true (see Example \ref{example_l-th_minimal}). If one does hope to establish a `converse' in Theorem~\ref{thm A} (b), then it is necessary to impose some additional conditions on components of those $\mathrm{ND}(\ell)$-schemes of `minimal degree of $\ell$-th kind' (i.e. $\deg(X)={e+\ell\choose \ell}$).
\end{remark}

As a consequence of Theorem \ref{thm A}, using the upper bound for $\beta_{i,\ell}(X)$ we can obtain a generalization of  a part of Green's $K_{p,1}$-theorem on the linear strand by quadrics of nondegenerate varieties in \cite{G2} to case of the first non-trivial linear strand by higher degree equations of any $\mathrm{ND}(\ell)$-schemes as follows.

\begin{corollary}[$K_{p,\ell}$-theorem for $\mathrm{ND}(\ell)$-subscheme] Let $X$ be any closed subscheme of codimension $e$ in $\P^{n+e}$ satisfying condition $\mathrm{ND}(\ell)$. Then, $\beta_{i,j}(X)=0$ for each $i > e,~j\le\ell$.
\end{corollary}

\begin{remark}[Characteristic $p$ case]\label{remk_ch_p}
Although we made the assumption that the base field $\kk$ has characteristic zero at the beginning of this section, most of results in the section still hold outside of low characteristics; see \cite[theorem~15.23]{E1}. For instance, 
Theorem~\ref{thm A} holds for any characteristic $p$ such that $p>\reg(I_X)$, where $\reg(I_X)$ is equal to the maximum of degrees of monomial generators in $\Gin(I_X)$ with respect to the degree reverse lexicographic order.  
\end{remark}

\section{Property $\N_{d,p}$ and Syzygies}\label{sect_Ndp}

\subsection{Geometry of property $\N_{d,p}$}

In this subsection, we assume that the base field $\kk$ is algebraically closed of any characteristic. We obtain two geometric implications of property $\mathbf{N}_{d,p}$ via projection method and the elimination mapping cone sequence; see \cite{AK2,HK2}. For the remaining of the paper, we call a reduced projective scheme $X\subset\P^N$ an \ti{algebraic set} (see also \cite[chapter 5]{E2}).

\begin{thm}[Syzygetic B\'ezout theorem]\label{thm:20160115-01}
Let $X\subset \P^{n+e}$ be a non-degenerate algebraic set of dimension $n$ satisfying $\N_{d,p}$
with $2\le d$ and $p\leq e$. Suppose that $L \subset \mathbb P^{n+e}$ is any linear space of
dimension $p$ whose intersection with $X$ is zero-dimensional. Then
\begin{itemize}
\item [(a)] $\length(L\cap X)\le \binom{d-1+p}{p}.$
\item [(b)] Moreover, if $\length(L\cap X)=\binom{d-1+p}{p}$, then for $1\le k\le d-1$ the base locus of a linear system $|H^0(\mathcal I_{X/\mathbb P^{n+e}}(k))|$ contains the multisecant space $L$ .
\end{itemize}
\end{thm}

\begin{remark} We would like to make some remarks on this result as follows:
\begin{itemize}
\item [(a)] If $p=1$ then it is straightforward by B\'ezout's theorem. Thus, Theorem \ref{thm:20160115-01} can be regarded as a syzygetic generalization to multisecant linear spaces when $p\ge 2$. 
\item [(b)] Note that in the theorem the length bound itself can be also obtained from \cite[theorem 1.1]{EGHP1}. We provide an alternative proof on it using geometric viewpoint of projection and further investigate the situation in which the equality holds.
\end{itemize}
\end{remark}

\begin{proof}[Proof of Theorem \ref{thm:20160115-01}]
(a) It is obvious when $p=1$. Now, let $X$ be an algebraic set satisfying the property $\N_{d,p}, p\ge 2$ and suppose that
$L \subset \mathbb P^{n+e}$ is a linear space of dimension $p$ whose intersection with $X$
is zero-dimensional.

Choose a linear subspace $\Lambda \subset L$ of dimension $p-1$ with homogeneous coordinates
$x_0, x_1, \ldots, x_{p-1}$  such that $X\cap \Lambda =\emptyset$. Consider a projection
$\pi_{\Lambda}: X \to \pi_{\Lambda}(X)\subset \P^{n+e-p}$. Then, $L\cap X$ is a fiber of
$\pi_{\Lambda}$ at the point $\pi_{\Lambda}(L\setminus \Lambda)\in \pi_{\Lambda}(X)$.
The key idea is to consider the syzygies of $R/I_X$ as an $S_p=\kk[x_p,\ldots,x_{n+e}]$-module
which is the coordinate ring of $\P^{n+e-p}$.
By \cite[corollary~2.4]{AK2}, $R/I_X$ satisfies ${\textup{\textbf {N}}}^{S_p}_{d, 0}$ as an $S_p=\kk[x_p,\ldots,x_{n+e}]$-module, i.e. we have the following surjection
\begin{equation}\label{eq:20160115-101}
S_{p}\oplus S_{p}(-1)^{p}\oplus S_{p}(-2)^{\beta^{S_{p}}_{0,2}} \oplus \cdots\oplus S_{p}(-d+1)^{\beta^{{S_p}}_{0,d-1}}
\st{\varphi_0} R/I_X \to 0.
\end{equation}
Sheafifying \eqref{eq:20160115-101},
we have
$$
\cdots\rightarrow \mathcal O_{\P^{n+e-p}}\oplus\mathcal O_{\P^{n+e-p}}(-1)^{p}\oplus
\mathcal O_{\P^{n+e-p}}(-2)^{\beta^{S_p}_{0,2}}\oplus \cdots \oplus \mathcal O_{\P^{n+e-p}}(-d+1)^{\beta^{S_p}_{0,d-1}} \st{\widetilde{\varphi_p}}\pi_{{\Lambda}_{*}}{\mathcal O_X} \to 0.
$$
Say $q=\pi_{\Lambda}(L\setminus \Lambda)$. By tensoring $\mathcal O_{\P^{n+e-p}}(d-1)\otimes \kk(q)$, we have the surjection on vector spaces:
\begin{equation}\label{eq:20160114-102}
\left[\bigoplus_{0\le i\le d-1} \mathcal O_{\P^{n+e-p}}(d-1-i)^{\beta^{S_{p}}_{0,i}}\right]
\otimes \kk(q)\twoheadrightarrow H^0(\langle \Lambda, q\rangle, {\mathcal O}_{{\pi_{\Lambda}}^{-1}(q)}(d-1)).
\end{equation}
Note that by [\textit{ibid.} corollary 2.5] $\beta^{S_{p}}_{0,i}\leq \binom{p-1+i}{i}=h^0(\mathcal O_{\Lambda}(i))$ for $0\le i\le d-1$ in \eqref{eq:20160114-102}. So we have
$$\dim_{\kk}H^0(\langle \Lambda, q\rangle, {\mathcal O}_{{\pi_{\Lambda}}^{-1}(q)}(d-1))=\text{length}(L\cap X)\le \sum_{i=0}^{d-1} \beta^{S_p}_{0,i}\leq \sum_{i=0}^{d-1} \binom{p-1+i}{i}=\binom{d-1+p}{p}.$$

(b) Now assume that $\text{length}(L\cap X)=\binom{d-1+p}{p}$.
From the above inequalities, we see that $\beta^{S_{p}}_{0,i}=\binom{p-1+i}{i}$ for every $i$.
Hence the map in \eqref{eq:20160114-102} is an isomorphism. Thus, there is no equation
of degree $d-1$ vanishing on ${\pi_{\Lambda}}^{-1}(q)\subset L=\langle \Lambda, q\rangle$
(i.e. $H^0(\mathcal I_{{\pi_{\Lambda}}^{-1}(q)/L}(d-1))=0$~).
So, if $F\in H^0(\mathcal I_{X/\mathbb P^{n+e}}(k))$ for $2\le k\le d-1$, then $F|_{L}$ vanishes on ${\pi_{\Lambda}}^{-1}(q)\subset L$ and this implies that $F|_{L}$ is identically zero. Thus, $L$ is contained in $Z(F)$, the zero locus of $F$ as we claimed.
\end{proof}


Now, we think of another effect of property $\textbf{N}_{d,p}$ on loci of $d$-secant lines. For this purpose, let us consider an outer projection $\pi_{q}:X\to \pi_{q}(X)\subset \mathbb P^{n+e-1}=\Proj(S_1)$, $S_1=\kk[x_1, x_{2},\ldots,x_{n+e}]$ from a point $q=(1,0,\cdots,0)\in \big(\Sec(X)\cup \Tan(X)\big)\setminus X$. We are going to consider the locus on $X$ engraved by $d$-secant lines passing through $q$ via partial elimination ideals (abbr. PEIs) theory as below. 

When $f\in (I_X)_m$ has a leading term $\ini(f)=x_0^{d_0}\cdots x_{n+e}^{d_{n+e}}$ in the lexicographic order, we set $d_{x_0}(f)=d_0$,
the leading power of $x_0$ in $f$. Then it is well known (e.g. \cite[subsection 2.1]{HK2}) that $K_0(I_X):=\bigoplus_{m\ge 0}\big\{f \in (I_X)_m \mid d_{x_0}(f)=0\big\}=I_X\cap S_1$ is the saturated ideal defining $\pi_{q}(X)\subset \mathbb P^{n+e-1}$.

Let us recall some definitions and basic properties of partial elimination ideals (see also e.g. \cite[chap. 6]{G} or \cite{HK2} for details).
\begin{definition}[Partial Elimination Ideal]\label{def_partial elimi ideals}
Let $I\subset R$ be a homogeneous ideal and let  \[\tilde{K}_i(I)=\bigoplus_{m\ge 0}\big\{f\in I_{m}\mid d_{x_0}(f)\leq i\big\}.\]
If $f\in \tilde{K}_i(I)$, we may write uniquely $f=x_0^i\bar{f}+g$ where $d_{x_0}(g)<i$ and define $K_i(I)$ by the image of $\tilde{K}_i(I)$
in $S_1$ under the map $f\mapsto \bar{f}$. We call $K_{i}(I)$ the $i$-th partial elimination ideal of $I$.
\end{definition}
Note that $K_{0}(I)=I\cap S_1$ and there is a short exact sequence as graded $S_1$-modules
\begin{equation}\label{eq:201}
0\rightarrow \frac{\tilde{K}_{i-1}(I)}{\tilde{K}_{0}(I)}
\rightarrow \frac{\tilde{K}_{i}(I)}{\tilde{K}_{0}(I)}\rightarrow
K_{i}(I)(-i)\rightarrow 0.
\end{equation}
In addition, we have the filtration on partial elimination ideals of $I$:
\[K_0(I)\subset K_1(I)\subset K_2(I)\subset \cdots \subset K_i(I)\subset \cdots \subset S_1=\kk[x_1, x_{2},\ldots,x_{n+e}].\]

It is well-known that for $i\geq 1$, the $i$-th partial elimination ideal $K_i(I_X)$ set-theoretically defines
$$Z_{i+1}:=\big\{y \in \pi_{q}(X)\mid \mult_y(\pi_{q}(X))\ge i+1\big\}$$

(e.g. \cite[proposition 6.2]{G}). Using this PEIs theory, we can describe the $d$-secant locus $$\Sigma_{d}(X):=\{x\in X \mid {\pi_q}^{-1}(\pi_{q}(x)) \text { has length d} ~\}$$ as a hypersurface $F$ of degree $d$ in the linear span $\big\langle F, q\big\rangle$ provided that $X$ satisfies $\N_{d,2}~(d\ge 2)$.

\begin{thm}[Locus of $d$-secant lines]\label{Loci_d_secant_lines}
Let $X\subset \mathbb P^{n+e}$ be a nondegenerate integral variety of dimension $n$ satisfying
$\N_{d,2}~(d\ge 2)$. For a projection $\pi_{q}:X\to \pi_{q}(X)\subset \P^{n+e-1}$ where $q\in \big(\Sec(X)\cup \Tan(X)\big)\setminus X$,
consider the $d$-secant locus $\Sigma_{d}(X)$. Then, we have
\begin{itemize}
\item[(a)] $\Sigma_d(X)$ is either empty or a hypersurface $F$ of degree $d$ in the linear span $\big\langle F, q\big\rangle$;
\item[(b)] $Z_{d}=\pi_{q}(\Sigma_{d}(X))$ is either empty or a linear subspace in $\pi_{q}(X)$ parametrizing the locus of $d$-secant lines through $q$;
\item[(c)] For a point $q\in \Sec(X)\setminus \big(\Tan(X)\cup X\big)$, there is a unique $d$-secant line through $q$ if $Z_{d}\neq \emptyset$.
\end{itemize}
\end{thm}
\begin{proof}
(a): Since $R/I_X$ satisfies $\textbf{N}_{d,2}$, it also satisfies $\textbf{N}_{d,1}$ as an $S_1$-module and we have the following exact sequence~:
$$\rightarrow\cdots \rightarrow \bigoplus^{d-1}_{j=1}S_1(-1-j)^{\beta_{1,j}^{S_1}} \st{\varphi_1} \bigoplus^{d-1}_{i=0}S_1(-i) \st{\varphi_0}R/I_X \rightarrow  0.
$$
Furthermore, $\ker \varphi_0$ is just $\tilde{K}_{d-1}(I_X)$ and we have a surjection
$$\cdots\rightarrow \bigoplus^{d-1}_{j=1}S_1(-1-j)^{\beta_{1,j}^{S_1}} \st{\varphi_1} \tilde{K}_{d-1}(I_X)\rightarrow 0.$$
Therefore, $\tilde{K}_{d-1}(I_X)$ is generated by elements of at most degree $d$.

Now consider the following commutative diagram of $S_1$-modules with  $K_{0}(I_X)=I_X\cap S_1$:
\begin{equation}\label{diagram:302-1}
\begin{array}{ccccccccccccccccccccccc}
&&0&&0&&0&&\\[1ex]
&&\downarrow &&\downarrow&&\downarrow&&\\[1ex]
0&\rightarrow& K_{0}(I_X) & \rightarrow & S_1& \rightarrow &
S_1/K_{0}(I_X)&\rightarrow & 0 \\[1ex]
&&\downarrow && \downarrow && \,\,\,\downarrow{\tilde \alpha} & \\[1ex]
0&\rightarrow& \tilde{K}_{d-1}(I_X) & \rightarrow &
\oplus_{i=0}^{d-1} S_1(-i)& \st{\varphi_0} & R/I_X
&\rightarrow & 0\\[1ex]
&&\downarrow &&\downarrow && \downarrow &&\\[1ex]
0&\rightarrow& {\tilde{K}_{d-1}(I_X)}/K_{0}(I_X) & \rightarrow &
\oplus_{i=1}^{d-1} S_1(-i)& \rightarrow& \coker\,\, \tilde \alpha &\rightarrow &0\\[1ex]
&&\downarrow &&\downarrow&&\downarrow&&\\[1ex]
&&0&&0&&0&&
\end{array}
\end{equation}

From the left column sequences in the diagram~(\ref{diagram:302-1}), $\tilde{K}_{d-1}(I_X)/{K_0(I_X)}$ is also generated
by at most degree $d$ elements.
On the other hands, we have a short exact sequence from (\ref{eq:201}) :
\begin{equation}\label{diagram:partial}
0\rightarrow \frac{\tilde{K}_{d-2}(I_X)}{K_0(I_X)} \rightarrow \frac{\tilde{K}_{d-1}(I_X)}{K_0(I_X)}\rightarrow
K_{d-1}(I_X)(-d+1)\rightarrow 0,
\end{equation}
Hence, $K_{d-1}(I_X)$ is generated by at most linear forms. So, $Z_{d-1}$ is either empty or a linear space.
Since $\pi_{q}:\Sigma_d(X)\twoheadrightarrow Z_{d}\subset \pi_{q}(X)$ is a $d:1$ morphism, $\Sigma_q(X)$ is a hypersurface of degree $d$
in $\big\langle Z_{d-1}, q\big\rangle$.  For a proof of (c), if $\dim \Sigma_d(X)$ is positive, then clearly, $q\in \Tan\,\Sigma_q(X)\subset \Tan(X)$.
So,  $\dim \Sigma_d(X)=\dim Z_{d}=0$ and there is a unique $d$-secant line through $q$.
\end{proof}

In particular, in the case of $d=2$, \textit{entry locus} of $X$ (i.e. locus of 2-secant lines through an outer point) is a quadric hypersurface, which was very useful to classify non-normal del Pezzo varieties in \cite{BP} by Brodmann and Park.

\subsection{Syzygetic rigidity for $d$-regularity}

In particular, if $p=e$ then we have the following corollary of Theorem \ref{thm:20160115-01} with characterization of the extremal cases.

\begin{corollary}\label{cor:bezout}
Let $X\subset \P^{n+e}$ be any non-degenerate algebraic set over an algebraically
closed field $\kk$ of characteristic zero. Suppose that $X$ satisfies $\mathbf {N}_{d,e}$ for some $d\ge 2$. Then, we have
$$\deg(X)\le \binom{d-1+e}{e}$$
and the following are equivalent:
\begin{itemize}
\item[(a)] $\deg(X)=\binom{d-1+e}{e}$;
\item[(b)] $X$ is arithmetically Cohen-Macaulay~(ACM) with $d$-linear resolution.
\end{itemize}
\end{corollary}
\begin{proof}
 It suffices to show that (a) implies (b). By the assumption that $\deg(X)$ is maximal,
 $\text{length}(L\cap X)=\binom{d-1+e}{e}$ for a generic linear space $\Lambda$ of dimension $e$.
 From a proof of Theorem~\ref{thm:20160115-01}, we see that there is no equation
of degree $d-1$ vanishing on ${\pi_{\Lambda}}^{-1}(q)\subset L=\langle \Lambda, q\rangle$
(i.e. $H^0(\mathcal I_{{\pi_{\Lambda}}^{-1}(q)/L}(d-1))=0$~). This means $X$ satisfies ND$(d-1)$
condition. In particular, it follows from Theorem~\ref{thm A} (a) that
$\beta_{e,d-1}(X)\le{e+d-2\choose d-1}.$

We also see from \cite[corollary~2.4]{AK2} that $\beta^{S_e}_{0,d-1}\leq \beta^R_{e,d-1}=\beta_{e,d-1}(X)$
because $X$ satisfies $\N_{d,e}$.
Note that $\beta^{S_{e}}_{0,d-1}=\binom{e+d-2}{d-1}=h^0(\mathcal O_{\Lambda}(d-1))$ in \eqref{eq:20160114-102}.
Therefore, $$\ds\beta_{e,d-1}(X)={e+d-2\choose d-1}.$$
So, we conclude from Theorem~\ref{thm A} (b) that $X$ is ACM with $d$-linear resolution.
\end{proof}

\begin{remark}
The above corollary can also be proved by the generalized version of the multiplicity conjecture which was shown by Boij-S\"{o}derberg \cite{BoSo}. Not relying on Boij-S\"{o}derberg theory, here we give a geometric proof for the multiplicity conjecture in this special case.
\end{remark}

As a consequence of previous results, now we can derive a syzygetic rigidity for $d$-regularity as follows:

\begin{thmB_again}[Syzygetic rigidity for $d$-regularity] Let $X\subset\P^{n+e}$ be any algebraic set of codimension $e$ over an algebraically
closed field $\kk$ of $\ch(\kk)=0$ satisfying condition $\mathrm{ND}(d-1)$ for some $d\ge2$. If $X$ has property $\N_{d,e}$, then $X$ is $d$-regular (more precisely, $X$ has ACM $d$-linear resolution).
\end{thmB_again}
\begin{proof}
By Theorem \ref{thm A} and Corollary \ref{cor:bezout}, if $X$ satisfies both condition $\mathrm{ND}(d-1)$ and property $\mathbf{N}_{d,e}$, then the degree of $X$ should be equal to $\binom{d-1+e}{e}$ and this implies that $X$ has ACM $d$-linear resolution (in particular, $X$ is $d$-regular).
\end{proof}

We would like to note that Theorem \ref{syz_rigid} does not hold without condition $\mathrm{ND}(\ell)$ even though the given $X$ is an \textit{irreducible} variety.

\begin{example}[Syzygetic rigidity fails without condition $\mathrm{ND}(\ell)$]\label{non_ex_syz_rigidity}
Let $\mathbf{d}=(d_0,\ldots,d_s)$ be a strictly increasing sequence of integers and $\mathbb{B}(\mathbf{d})$ be the pure Betti table associated to $\mathbf{d}$; see \cite{BoSo}. Due to Boij-S\"{o}derberg theory, we can construct a Betti table  $\mathbb{B}_0$ as given by
\begin{center}
\texttt{\begin{tabular}{l|ccccc}
          & 0 & 1   &  2   &3&4    \\ \hline
    0    & 1 & -    & -   &-&-     \\
    1    & - & -    & -   &-&-     \\
    2    & - & -    & -   &-&-     \\
    3   & -  & 18 & 32 &16&-\\
    4    & - & -    & -   &-&1    
    \end{tabular}~,}
    \end{center}
    from the linear combination $\frac{4}{5}\mathbb{B}\big((0,4,5,6)\big)+\frac{1}{5}\mathbb{B}\big((0,4,5,6,8)\big)$. This $\mathbb{B}_0$ expects a curve $C$ of degree 16 and genus 13 in $\P^4$ with $h^1(\mathcal{O}_C(1))=1$ (i.e. $e=3$), which satisfies property $\N_{4,e}$, but not $4$-regular (i.e. Theorem \ref{syz_rigid} fails). This Betti  table can be realized as the one of a projection $C$ into $\P^4$ of a canonically embedded genus $13$ general curve $\widetilde{C}\subset\P^{12}$ from random $8$ points of $\widetilde{C}$. Note that $C$ is irreducible (in fact, smooth) and has no defining equations of degree less than $4$, but is not $\mathrm{ND}(3)$-curve because $\deg(C)=16\ngeq{3+3\choose3}=20$. Here is a \texttt{Macaulay 2} code for this:
    \smallskip
{\footnotesize
\begin{Verbatim}
loadPackage("RandomCanonicalCurves",Reload=>true);
setRandomSeed("alpha");
g=13; k=ZZ/32003;
S=k[x_0..x_(g-1)];
I=(random canonicalCurve)(g,S);
for i from 0 to 7 do P_i=randomKRationalPoint I;
L=intersect apply(8,i->P_i); R=k[y_0..y_4];
f=map(S,R,super basis(1,L));
RI=preimage(f,I); betti res RI
\end{Verbatim}
}
\end{example}

\section{Comments and Further Questions}\label{sect_question}

\noindent In the final section, we present some relevant examples and discuss a few open questions related to our main results in this paper.\\

\noindent{\textsf{I. Certificates of condition $\mathrm{ND}(\ell)$~}} First of all, from the perspective of this article, it would be very interesting to provide more situations to guarantee condition $\mathrm{ND}(\ell)$. As one way of thinking, one may ask where condition $\mathrm{ND}(\ell)$ does hold largely. For instance, as discussed in Remark \ref{NDk_in_large}, we can consider this problem  as follows: 

\begin{question}\label{Q_deg_bd}  
For given $e,\ell>0$, is there a function $f(e,\ell)$ such that any $X\subset\P^{n+e}$ of codimension $e$ is $\mathrm{ND}(\ell)$-subscheme if $\deg(X)>f(e,\ell)$ and $H^0(\mathcal{I}_{X/\P^{n+e}}(\ell))=0$?
\end{question}

We showed that there are positive answers for this question in case of codimension two in Proposition \ref{codim2_ND_k} and Example \ref{sp_curve_ND_k}. What about in \textit{higher codimensional} case? (recall that a key ingredient for Proposition \ref{codim2_ND_k}  is `lifting theorem' which is well-established in codimension 2)

The following example tells us that for Question \ref{Q_deg_bd} one needs to assume irreducibility or some conditions on irreducible components of $X$ in general. 

\begin{example}[A non-$\mathrm{ND}(2)$ reduced scheme of arbitrarily large degree] Consider a closed subscheme $X\subset \P^3$ of codimension $2$ defined by the monomial ideal $I_X=(x_0^3, x_0^2x_1, x_0x_1^2, x_1^t, x_0^2x_2)$ for any positive integer $t\geq 4$. Note that $h^0(\mathcal I_{X/\P^3}(2))=0$ and $\deg(X)=t+2\geq 6=\binom{e+2}{2}.$ 
Since $I_X$ is a Borel fixed monomial ideal, we see that $I_{X\cap L/ \P^3}=(x_0^2, x_0x_1^2, x_1^t)$ for a general linear form $L$, which implies that $X$ does not satisfy $\mathrm{ND}(2)$.

If we consider a sufficiently generic distraction $D_{\mathcal L}(I_X)$ of $I_X$ (see \cite{BCR} for details of distraction), then it is of the form
$$D_{\mathcal L}(I_X) = (L_1L_2L_3, L_1L_2L_4, L_1L_4L_5, \prod_{j=1}^{t}{M_j}, L_1L_2L_7)~,$$
where $L_i$ and $M_j$ are generic linear forms for each $1\leq i\leq 7$ and $1\leq j\leq t$. Then $D_{\mathcal L}(I_X)$ defines the union of $t+2$ lines and $3$ points. Using this, we can construct an example of non-$\mathrm{ND}(2)$ algebraic set of arbitrarily large degree.
\end{example}
\medskip

\noindent{\textsf{II. Condition $\mathrm{ND}(\ell)$ and non-negativity of $h$-vector~}} For any closed subscheme $X\subset\P^{n+e}$ of dimension $n$, the Hilbert series of $R_X:=\kk[x_0,\ldots,x_{n+e}]/I_X$ can be written as
\begin{equation}\label{h_series}
H_{R_X}(t)= \sum \big(\dim_\kk (R_X)_i \big) t^i = \frac{h_0+h_1t+ \cdots + h_s t^s}{(1-t)^{n+1}}
\end{equation}
and the $h$-vector $h_0, h_1, \ldots, h_s$ usually contains much information on the coordinate ring $R_X$ and on geometric properties of $X$. One of the interesting questions on the $h$-vector is the one to ask about \textit{non-negativity} of the $h_i$ and it is well-known that every $h_i\ge0$ if $R_X$ is Cohen-Macaulay (i.e. $X$ is ACM). Recently, a relation between Serre's condition $(S_\ell)$ on $R_X$ and non-negativity of $h$-vector has been focused as answering such a question as 
\begin{center}
`~Does Serre's condition $(S_\ell)$ imply $h_0, h_1, \ldots, h_\ell\ge0$~?~'
\end{center}
This was checked affirmatively in case of $I_X$ being a square-free monomial ideal by Murai and Terai \cite{MT}. More generally, in \cite{DMV} Dao, Ma and Varbaro proved the above question is true under some mild singularity conditions on $X$ (to be precise, $X$ has Du Bois singularity in $\ch(\kk)=0$ or $R_X$ is F-pure in $\ch(\kk)=p$). Here, we present an implication of condition $\mathrm{ND}(\ell)$ on this question as follows.

%

\begin{corollary}[$\ND(\ell)$ implies non-negativity of $h$-vector]\label{ND_hvec}
Let $X=\Proj(R_X)$ be any closed subscheme of codimension $e$ in $\P^{n+e}$ over an algebraically closed field $\kk$ with $\ch(\kk)=0$ and $h_i$'s be the $h$-vector of $R_X$ in (\ref{h_series}). Suppose that $X$ has condition $\ND(\ell-1)$. Then, $h_0, h_1, \ldots, h_\ell\ge0$.
\end{corollary}
\begin{proof}
Say $r_i=\dim_\kk (R_X)_i$. First of all, by (\ref{h_series}), we have
$$(1-t)^{n+1}\big(r_0+r_1 t+r_2 t^2 +\cdots \big)=h_0 + h_1 t +h_2 t^2 +\cdots~,$$
which implies that $h_0=r_0,~ h_1={n+1\choose 1}(-1) r_0 +r_1,~\cdots, ~ h_j=\sum_{i=0}^{j} {n+1\choose i}(-1)^i r_{j-i}$ for any $j$. Since $r_{j-i}={n+e+j-i\choose j-i}-\dim_\kk (I_X)_{j-i}$, it holds that
\begin{align*}
h_j&=\sum_{i=0}^{j} {n+1\choose i}(-1)^i {n+e+j-i\choose j-i}-\sum_{i=0}^{j}{n+1\choose i}(-1)^i \dim_\kk (I_X)_{j-i} \\
&={e+j-1\choose j}-\sum_{i=0}^{j}{n+1\choose i}(-1)^i \dim_\kk (I_X)_{j-i}\quad\cdots\quad(\ast)~,
\end{align*}
where the last equality comes from comparing $j$-th coefficients in both sides of the identity $$\ds (1-t)^{n+1}\bigg[\sum_{i\ge0}{n+e+i\choose i}t^i \bigg]=\frac{1}{(1-t)^e}~.$$ 

Now, by Theorem \ref{thm A} (a), we know that $\dim_\kk (I_X)_{0}=\dim_\kk (I_X)_{1}=\cdots=\dim_\kk (I_X)_{\ell-1}=0$ and $\dim_\kk (I_X)_{\ell}\le{e+\ell-1\choose \ell}$. So, for any $j\le \ell-1$, by $(\ast)$ we see that $h_j={e+j-1\choose j}\ge0$. Similarly, we obtain that $h_\ell={e+\ell-1\choose \ell}-\dim_\kk (I_X)_{\ell}\ge0$ as we wished.
\end{proof}

Hence, it is natural to ask:

\begin{question}\label{Q_Serre} How are Serre's $(S_\ell)$ on $R_X$ and condition $\mathrm{ND}(\ell)$ on $X$ related to each other?
\end{question}
For example, it would be nice if one could find some implications between the notions under reasonable assumptions on singularities or connectivity of components.

\medskip

\noindent{\textsf{III. Geometric classification/characterization of ACM $d$-linear varieties~} For further development, it is natural and important to consider the boundary cases in Theorem~\ref{thm A} from a \textit{geometric} viewpoint. When $\ell=1$, due to del Pezzo-Bertini classification, we completely understand the extremal case, that is ACM $2$-linear varieties, geometrically; (a cone of) quadric hypersurface, Veronese surface in $\P^5$ or rational normal scrolls. It is also done in category of algebraic sets in \cite{EGHP2}. What about ACM varieties having \textit{$3$-linear} resolution? or \textit{higher $d$-linear} resolution? The followings are first examples of variety with ACM $3$-linear resolution.

\begin{example}[Varieties having ACM $3$-linear resolution]\label{first_example_ACM3Lin} We have
\begin{itemize}
\item[(a)] Cubic hypersurface ($e=1$);
\item[(b)] 3-minors of $4\times 4$ generic symmetric matrix (i.e. the secant line variety $Sec(v_2(\P^3))\subset\P^9$);
\item[(c)] 3-minors of $3\times (e+2)$ sufficiently generic matrices (e.g. secant line varieties of rational normal scrolls);
\item [(d)] $\Sec(v_3(\P^2))$; $\Sec(\P^2 \times \P^1\times \P^1)$;
\end{itemize}
\end{example}

Most of above examples come from taking secants. Unless a hypersurface, are they all the secant varieties of relatively small degree varieties? Recall that any secant variety $Sec(X)$ not equal to the ambient space is always `singular' because $\Sing(Sec(X))\supset X$. But, we can construct examples of \textit{smooth} $3$-linear ACM of low dimension as follows:
 
 \begin{example}[Non-singular varieties with ACM $3$-linear resolution]\label{example_smooth_3-linearACM}
We have 
\begin{itemize}
\item[(a)] (A non-hyperelliptic low degree curve of genus 3 in $\P^3$) For a smooth plane quartic curve $C$ of genus $g=3$. One can re-embed $C$ into $\P^9$ using the complete linear system $|\mathcal O_C(3)|$. Say this image as $\widetilde{C}$. For $\deg\widetilde{C}=12$, $\widetilde{C}\subset\P^9$ satisfies at least property $\mathbf{N}_5$ by the Green-Lazarsfeld theorem. We also know that
\[H^0(\mathcal I_{\widetilde{C}}(2))=H^0(\mathcal O_{\P^9}(2))-H^0(\mathcal O_{\widetilde{C}}(2))={9+2\choose2}-(2\cdot12+1-3)=55-22=33~.\]
Now, take any 6 smooth points on $\widetilde{C}$ and consider inner projection of $\widetilde{C}$ from these points into $\P^3$. Denote this image curve in $\P^3$ by $\overline{C}$. From \cite[proposition 3.6]{HK}, we obtain that
\[H^0(\mathcal I_{\overline{C}}(2)=H^0(\mathcal I_{\widetilde{C}}(2))-(8+7+6+5+4+3)=33-33=0~.\]
In other words, there is no quadric which cuts out $\overline{C}$ in $\P^3$. Since $C$ is non-hyperelliptic, $\overline{C}$ is projectively normal (i.e. ACM). Therefore, $\overline{C}$ is a smooth $\mathrm{ND}(2)$-curve in $\P^3$ and has $\deg \overline{C}=6$ which is equal to $2g$. Using \texttt{Macaulay 2} \cite{M2}, we can also check all these computations including the minimal resolution of $\overline{C}\subset\P^3$. $\overline{C}$ has ACM 3-linear resolution such as \texttt{
\begin{tabular}{l|ccc}
          & 0 & 1   &  2       \\ \hline
    0    & 1 & -    & -        \\
    1   & -  & - & - \\
    2    & -  & 4 & 3
    \end{tabular}}\quad.
    
\item[(b)] (A surface in $\P^6$) Consider a rational normal surface scroll $X=S(4,4)$ in $\P^9$. Its secant line variety $Y=Sec(X)$ is a $5$-fold and has a minimal free resolution as $$\texttt{\begin{tabular}{l|ccccc}
          & 0 & 1   &  2   & 3 & 4    \\ \hline
    0    & 1 & -    & -    & - & -     \\
    1   & -  & - & - & - & -\\
    2    & -  & 20 & 45 & 36 & 10
    \end{tabular}}\quad,$$ which is ACM $3$-linear. Even though $Y$ is singular, as we cut $Y$ by three general hyperplanes $H_1, H_2, H_3$ we obtain a smooth surface $S=Y\cap H_1 \cap H_2 \cap H_3$ of degree $15$ in $\P^6$ whose resolution is same as above (one can check all the computations using \cite{M2}).
    \end{itemize}
\end{example}

It is interesting to observe that every variety of dimension $\ge2$ in Example \ref{first_example_ACM3Lin} and \ref{example_smooth_3-linearACM} has a \textit{determinantal} presentation for its defining ideal.
 
 \begin{question}\label{Q_geom_ACM_Lin}
Can we give a geometric classification or characterization of ACM $d$-linear varieties for $d\ge3$? Do they all come from (a linear section of) secant construction except very small (co)dimension? In particular, does it always have a determinantal presentation if  $X$ is ACM $3$-linear variety and $\dim X\ge2$?
\end{question}

Finally, we present some example as we discussed in Remark \ref{thm A_recover_HK}. 

\begin{example}[Minimal degree of $\ell$-th kind ($\ell\ge2$) does not guarantee ACM linear resolution]\label{example_l-th_minimal}
In contrast with $\ell=1$ case, a converse of Theorem \ref{thm A} (b)
\begin{center}
`the equality $\deg(X)={e+\ell\choose \ell}$ with $\mathrm{ND}(\ell)$ implies that $X$ has ACM $(\ell+1)$-linear resolution'
\end{center}
does not hold for $\ell\ge2$ (note that, in the case of classical minimal degree, the statement does hold under $\mathrm{ND}(1)$-condition once we assume irreducibility or some connectivity condition on components of $X$ such as `linearly joined' in \cite{EGHP2}). 

By manipulating Gin ideals and distraction method, one could generate many reducible examples of such kind. Even though $X$ is irreducible, we can construct a counterexample. As a small example, using \cite{M2} we can verify that a smooth rational curve $C$ in $\mathbb P^3$ of degree $6$, a (isomorphic) projection of a rational normal curve in $\P^6$ from 3 random points, has Betti table as in Figure \ref{betti_l-th_minimal}.
\begin{figure}[!htb]
\begin{subfigure}{.44\textwidth}
\centering
\texttt{\begin{tabular}{l|cccccccc}
          & 0 & 1   &  2  &3     \\ \hline
    0    & 1 & -    & -    &-    \\
    1   & -  & - & - &-\\
    2   & -  & 1 & - &-\\
    3    & -  & 6 & 9 & 3
    \end{tabular}}
\end{subfigure}
\begin{subfigure}{.44\textwidth}
\centering
\texttt{\begin{tabular}{l|ccccccc}
          & 0 & 1   &  2       \\ \hline
    0    & 1 & -    & -        \\
    1   & -  & - & - \\
    2   & -  & 4 & 3 \\
    3   & -  & -  & -
    \end{tabular}}
    \end{subfigure}~.
    \caption{Betti tables of $C$ and $C\cap H$}
    \label{betti_l-th_minimal}
\end{figure}

Note that $C$ satisfies condition $\ND(2)$ and is of minimal degree of 2nd kind (i.e. $\deg(C)={2+2\choose 2}$), but its resolution is still not 3-linear.
\end{example}

\end{document}